\documentclass[12pt]{article}
\usepackage{amssymb,latexsym,amsmath,enumerate,verbatim,amsfonts,amsthm}
\usepackage[active]{srcltx}
\numberwithin{equation}{section}
\newtheorem{theorem}{Theorem}[section]
\newtheorem{definition}{Definition}[section]
\newtheorem{proposition}[theorem]{Proposition}

\newtheorem{lemma}{Lemma}[section]
\newtheorem{corollary}[theorem]{Corollary}

\newtheorem{question}{Question}

\makeatletter
 \def\@biblabel#1{#1.}
\makeatother

\newcommand{\A}{{\cal A}}

\newcommand{\uu}{{\bf u}}

\newcommand{\x}{{\bf x}}
\newcommand{\y}{{\bf y}}
\newcommand{\z}{{\bf z}}
\newcommand{\q}{{\bf q}}

\newcommand{\e}{{\bf e}}
\newcommand{\0}{{\bf 0}}

\begin{document}
\title{Strictly semi-positive tensors and the boundedness of tensor complementarity problems}

\author{Yisheng Song\thanks{ Corresponding author. School of Mathematics and Information Science  and Henan Engineering Laboratory for Big Data Statistical Analysis and Optimal Control,
 Henan Normal University, XinXiang HeNan,  P.R. China, 453007.
Email: songyisheng@htu.cn. This author's work was supported by
the National Natural Science Foundation of P.R. China (Grant No.
11571095,11601134).  His work was partially done when he was
visiting The Hong Kong Polytechnic University.}, \quad Liqun Qi\thanks{ Department of Applied Mathematics, The Hong Kong Polytechnic University, Hung Hom,
Kowloon, Hong Kong. Email: maqilq@polyu.edu.hk. This author's work was supported by the Hong
Kong Research Grant Council (Grant No. PolyU 502111, 501212, 501913 and 15302114).}}

\date{\today}

 \maketitle

\begin{abstract}
\noindent  
In this paper, we present the boundedness of solution set of tensor complementarity problem defined by a strictly semi-positive tensor. For strictly semi-positive tensor, we prove that all $H^+(Z^+)$-eigenvalues of each principal sub-tensor are positive. We define two new constants associated with $H^+(Z^+)$-eigenvalues of a strictly semi-positive tensor.  With the help of these two constants, we establish upper bounds of an important quantity  whose positivity is a necessary and sufficient condition for a general tensor to be a strictly semi-positive tensor.  The monotonicity and boundedness of such a quantity are established too.

\noindent {\bf Key words:}\hspace{2mm} Strictly semi-positive tensor,  Tensor complementarity problem, upper and lower bounds,  Eigenvalues.  \vspace{3mm}

\noindent {\bf AMS subject classifications (2010):}\hspace{2mm}
47H15, 47H12, 34B10, 47A52, 47J10, 47H09, 15A48, 47H07.
  \vspace{3mm}

\end{abstract}
\section{Introduction}
\label{intro}
An $m$-order $n$-dimensional tensor (hypermatrix) $\mathcal{A} = (a_{i_1\cdots i_m})$ is a multi-array of real entries $a_{i_1\cdots
	i_m}$, where $i_j \in I_n=\{1,2,\cdots,n\}$ for $j \in I_m=\{1,2,\cdots,m\}$. Let $\x \in\mathbb{R}^n$. Then $\mathcal{A} \x^{m-1}$ is a vector in $\mathbb{R}^n$ with its $i$th component as
$$\left(\mathcal{A} \x^{m-1}\right)_i: = \sum_{i_2, \cdots, i_m=1}^n a_{ii_2\cdots
	i_m}x_{i_2}\cdots x_{i_m}$$ for $i \in I_n$.
Obviously, each component of $\mathcal{A} \x^{m-1}$ is a homogeneous polynomial of degree $m-1$. For any $\q\in\mathbb{R}^n$, we consider the tensor complementarity problem, a special class of nonlinear complementarity problems, denoted by TCP$(\mathcal{A},\q)$:  finding $\x \in \mathbb{R}^n$ such that 	$$ \x \geq \0, \q + \mathcal{A}\x^{m-1} \geq \0, \mbox{ and }\x^\top (\q + \mathcal{A}\x^{m-1}) = 0\leqno{\bf TCP(\mathcal{A},\q)}$$
or showing that no such vector exists. \\

Clearly, TCP$(\mathcal{A},\q)$ is a natural extension of linear complementarity problem ($m=2$). The notion of the tensor complementarity problem was used firstly by Song and Qi \cite{SQ-2015,SQ2015}.  Recently, Huang and Qi \cite{HQ} formulated an $n-$person noncooperative game as a tensor complementarity problem and showed that finding a
Nash equilibrium point of the multilinear game is equivalent to finding a solution of the tensor complementarity problem.    Very recently, the solution of TCP$(\mathcal{A},\q)$ and related problems have been well studied. Song and Qi \cite{SQ15} discussed the solution of TCP$(\mathcal{A},\q)$ with a strictly semi-positive tensor and proved the equivalence between  (strictly) semi-positive tensors and (strictly) copositive tensors in the case of symmetry.  Che, Qi, Wei \cite{CQW} discussed the existence and uniqueness of solution of TCP$(\mathcal{A},\q)$ with some special tensors. Song and Yu \cite{SY15} studied properties of the solution set of  the TCP$(\mathcal{A},\q)$ and obtained global upper bounds of the solution of  the TCP$(\mathcal{A},\q)$ with a strictly semi-positive tensor.  Luo, Qi and Xiu \cite{LQX} obtained the sparsest solutions to TCP$(\mathcal{A},\q)$ with a Z-tensor. Gowda, Luo, Qi and Xiu \cite{GLQX} studied the various
equivalent conditions of existence of solution to TCP$(\mathcal{A},\q)$ with a Z-tensor. Ding, Luo and Qi \cite{DLQ} showed the properties of TCP$(\mathcal{A},\q)$ with a P-tensor. Bai, Huang and Wang \cite{BHW} considered the global uniqueness and solvability for TCP$(\mathcal{A},\q)$ with a strong P-tensor.  Wang, Huang and Bai \cite{WHB} gave the solvability of TCP$(\mathcal{A},\q)$ with exceptionally regular tensors. Huang,Suo and Wang \cite{HSW} presented the several classes of  Q-tensors.
Song and Qi \cite{SQ13}, Ling, He, Qi \cite{LHQ2015,LHQ2016}, Chen, Yang, Ye \cite{CYY} studied the the tensor eigenvalue complementarity problem for higher order tensors.\\

The tensor complementarity problem, as a special type of nonlinear complementarity problems, is a new topic emerged from the tensor community, inspired by the growing research on structured tensors. At the same time, the tensor complementarity problem, as a natural extension of the linear complementarity problem seems to have similar  properties to such a problem, and to have its particular and nice properties other than ones of the classical nonlinear complementarity problem. So how to identify their good properties and applications will be very interesting by means of the special structure of  higher order tensors (hypermatrices).\\

Let $A = (a_{ij})$ be an $n \times n $ real matrix and $\q \in \mathbb{R}^n$. The linear complementarity problem, denoted by LCP$(A, \q)$, is to find $ \x \in \mathbb{R}^n$ such that
$$	\x \geq \0, \q + A\x \geq \0, \mbox{ and }\x^\top (\q + A\x) = 0
\leqno{\bf LCP(A,\q)}$$
or to show that no such vector exists. In past several decades, there have been  numerous mathematical workers concerned with the solution of LCP$(A, \q)$ by means of the special structure of the matrix $A$. An important topic of those studies is the error bound analysis for the solution of LCP$(A, \q)$. In 1990, Mathias and Pang \cite{MP90} discussed error bounds for LCP$(A, \q)$ with a P-matrix $A$.  Luo, Mangasarian, Ren, Solodov \cite{LMRS} established  error bounds for LCP$(A, \q)$ with a nondegenerate matrix. Chen and Xiang \cite{CX07,CX06} studied perturbation bounds of LCP$(A, \q)$ with a P-matrix and the computation of those error bounds.  Chen, Li, Wu, Vong \cite{CLWV} established error bounds of LCP$(A, \q)$ with a MB-matrix. Dai \cite{D2011} presented error bounds for LCP$(A, \q)$ with a DB-matrix. Dai, Li, Lu \cite{DLL12,DLL13} obtained   error bounds for LCP$(A, \q)$ with a SB-matrix. Garc\'ia-Esnaola and Pe\~na \cite{GP09} studied error bounds for LCP$(A, \q)$ with a B-matrix. Garc\'ia-Esnaola and Pe\~na \cite{GP12} proved error bounds for LCP$(A, \q)$ with a BS-matrix. Garc\'ia-Esnaola and Pe\~na \cite{GP10} gave the comparison of error bounds for LCP$(A, \q)$ with a H-Matrix. Recently, Li and Zheng \cite{LZ14} gave a new error bound for LCP$(A, \q)$ with a H-Matrix. Sun and Wang \cite{SW13} studied error bounds  for generalized linear complementarity problem under some proper assumptions. Wang and Yuan \cite{WY11} presented componentwise error bounds for LCP$(A, \q)$.\\

Motivated by the study on error bounds for LCP$(A, \q)$, we consider the boundedness of solution set for the tensor complementarity problem.
The following question is natural. May we extend the error bounds results  of the linear complementarity problem to  the tensor  complementarity problem with some class of specially structured tensors?\\

In this paper, we will mainly consider the above question. In order to showing the boundedness of solution set of tensor complementarity problem, we first study the properties of a quantity $\beta(\mathcal{A})$ for a strictly semi-positive tensor. Such  a quantity $\beta(\mathcal{A})$ closely adheres to the error bound analysis of TCP$(\mathcal{A},\q)$.  We show the monotonicity and boundedness of $\beta(\mathcal{A})$ and obtain that the strict positivity of $\beta(\mathcal{A})$ is equivalent to strict semi-positivity of a  tensor. We introduce two new constants associated with $H^+(Z^+)$-eigenvalues of a tensor and establish upper bounds of $\beta(\mathcal{A})$ for  strictly semi-positive tensor  $\mathcal{A}$. We also prove that all $H^+(Z^+)$-eigenvalues of each principal sub-tensor of a strictly semi-positive tensor are positive. Finally,  we present  the upper and lower bounds of solution set for TCP$(\mathcal{A},\q)$ with a strictly semi-positive tensor  $\mathcal{A}$.\\


We briefly describe our notation. Let $I_n:=\{1, 2,  \cdots, n\}.$  We
use small letters $x, u, v, \alpha, \cdots$, for scalars, small bold
letters $\x, \y, \cdots$, for vectors, capital letters $A, B,
\cdots$, for matrices, calligraphic letters $\mathcal{A}, \mathcal{B}, \cdots$, for tensors.   Denote the set of all
real $m$th order $n$-dimensional tensors by $T_{m, n}$ and the set of all
real $m$th order $n$-dimensional symmetric  tensors by $S_{m, n}$.   We
denote the zero tensor in $T_{m, n}$ by $\mathcal{O}$.   Denote $\mathbb{R}^n:=\{\x=(x_1, x_2,\cdots, x_n)^\top; x_i\in \mathbb{R}, i\in I_n\}$,  ${\mathbb{C}}^n:=\{(x_1,x_2,\cdots,x_n)^T;x_i\in {\mathbb{C}},  i \in I_n \}$, $\mathbb{R}^n_{+}=\{\x\in \mathbb{R}^n;\x\geq\0\}$, $\mathbb{R}^n_{-}=\{\x\in \mathbb{R}^n;\x\leq\0\}$ and $\mathbb{R}^n_{++}=\{\x\in \mathbb{R}^n;\x>0\}$, where $\mathbb{R}$ is the set of real numbers, and ${\mathbb{C}}$ is the set of complex numbers, $\x^\top$ is the transposition of a vector $\x$, and $\x\geq\0$ ($\x>\0$) means $x_i\geq0$ ($x_i>0$) for all $i\in I_n$.  Let $\e=(1,1,\cdots,1)^T$. Denote by $\e^{(i)}$ the
$i$th unit vector in $\mathbb{R}^n$, i.e., $e^{(i)}_j=1$ if $i=j$ and  $e^{(i)}_j=0$ if $i\neq j$, for $i,j\in I_n$.  For any vector $\x \in
\mathbb{C}^n$, $\x^{[m-1]}$ is a vector in $\mathbb{C}^n$ with
its $i$th component defined as $x_i^{m-1}$ for $i \in I_n$, and $\x \in
\mathbb{R}^n$, $\x_+$ is a vector in $\mathbb{R}^n$ with $(\x_+)_i=x_i$ if $x_i\geq0$ and $(\x_+)_i=0$ if $x_i<0$  for $i \in I_n$.  We assume that $m \ge 2$ and $n \ge 1$.  We denote by $\mathcal{A}^J_r$ the principal sub-tensor of a tensor $\mathcal{A} \in T_{m, n}$ such that the entries of  $\mathcal{A}^J_r$ are indexed by $J \subset I_n$ with $|J|=r$ ($1 \le r\leq n$), and denote by $\x_J$ the $r$-dimensional sub-vector of a vector $\x \in \mathbb{R}^n$, with the components of $\x_J$ indexed by $J$.

\section{Preliminaries and basic facts}

In this section, we will  collect some basic definitions and facts, which will be used later on.

All the tensors discussed in this paper are real. An $m$-order $n$-dimensional tensor (hypermatrix) $\mathcal{A} = (a_{i_1\cdots i_m})$ is a multi-array of real entries $a_{i_1\cdots
	i_m}$, where $i_j \in I_n$ for $j \in I_m$. If the entries $a_{i_1\cdots i_m}$ are
invariant under any permutation of their indices, then $\mathcal{A}$ is
called a {\bf symmetric tensor}. We now give the
definitions of  (strictly) semi-positive tensors (strictly) copositive tensors, which was introduced by Song and Qi \cite{SQ2015}.
The concept of (strictly) copositive tensors was first introduced and used by Qi in \cite{Qi1}.  Their equivalent definition and some special structures were proved by Song and Qi \cite{SQ-15}.
\begin{definition} \label{d21} 
	Let $\mathcal{A}  = (a_{i_1\cdots i_m}) \in T_{m, n}$.   $\mathcal{A}$ is said to be\begin{itemize}
		\item[(i)] {\bf semi-positive} iff for each $\x\geq0$ and $\x\ne\0$, there exists an index $k\in I_n$ such that $$x_k>0\mbox{ and }\left(\mathcal{A} \x^{m-1}\right)_k\geq0;$$
		\item[(ii)]  {\bf strictly semi-positive} iff for each $\x\geq\0$ and $\x\ne\0$, there exists an index $k\in I_n$ such that $$x_k>0\mbox{ and }\left(\mathcal{A} \x^{m-1}\right)_k>0,$$
		or equivalently, $$x_k\left(\mathcal{A} \x^{m-1}\right)_k>0;$$
		\item[(iii)] {\bf copositive } if $\mathcal{A}\x^m\geq0$ for all $\x\in \mathbb{R}^n_+$;
		\item[(iv)] {\bf strictly copositive} if  $\mathcal{A}\x^m>0$ for all $\x\in \mathbb{R}^n_+\setminus\{\0\}$.
		\item[(v)]  {\bf  Q-tensor} iff the TCP$(\mathcal{A},\q)$ has a  solution for all $\q \in \mathbb{R}^n$.
	\end{itemize}
\end{definition}
The following are two basic conclusions in the study of (strictly) semi-positive tensors.
\begin{lemma}\label{le:21}{\em(Song and Qi \cite[Corollary 3.3, Theorem 3.4]{SQ2015} and \cite[Theorem 3.3, 3.4]{SQ15})} Each strictly semi-positive tensor must be  a Q-tensor. If $\mathcal{A} \in S_{m, n}$, then $\mathcal{A}$ is  (strictly) semi-positive if and only if it is  (strictly) copositive.
\end{lemma}

The concept of principal sub-tensors was introduced and used in \cite{LQ1} for symmetric tensors.
\begin{definition} \label{d22}  Let $\mathcal{A}  = (a_{i_1\cdots i_m}) \in T_{m, n}$.
	A tensor $\mathcal{C} \in T_{m, r}$  is called {\bf a principal sub-tensor}  of a tensor $\mathcal{A} = (a_{i_1\cdots i_m}) \in T_{m, n}$ ($1 \le r\leq n$) iff there is a set $J$ that is composed of $r$ elements in $I_n$ such that
	$$\mathcal{C} = (a_{i_1\cdots i_m}),\mbox{ for all } i_1, i_2, \cdots, i_m\in J.$$
	Denote such a principal sub-tensor $\mathcal{C}$ by $\mathcal{A}_r^J$. \end{definition}


\begin{lemma} \label{le:22} \cite[Proposition 2.1, 2.2]{SQ15}
	Let $\mathcal{A}\in T_{m,n}$. Then
	\begin{itemize}
		\item[(i)] $a_{ii\cdots i}\geq0$ for all $i\in I_n$ if  $\mathcal{A}$ is semi-positive;
		\item[(ii)] $a_{ii\cdots i}>0$ for all $i\in I_n$ if  $\mathcal{A}$ is strictly semi-positive;
		\item[(iii)] there exists $k\in I_n$ such that  $\sum\limits_{i_2, \cdots, i_m=1}^n a_{ki_2\cdots
			i_m}\geq0$ if  $\mathcal{A}$ is semi-positive;
		\item[(iv)] there exists $k\in I_n$ such that  $\sum\limits_{i_2, \cdots, i_m=1}^n a_{ki_2\cdots
			i_m}>0$ if  $\mathcal{A}$ is strictly semi-positive;
		\item[(v)] each principal sub-tensor of a semi-positive tensor is  semi-positive;
		\item[(vi)] each principal sub-tensor of a strictly semi-positive tensor is strictly semi-positive.
	\end{itemize}
\end{lemma}
The concepts of tensor eigenvalues were introduced by Qi \cite{LQ1,LQ2} to the higher order symmetric tensors, and the existence of the eigenvalues and some applications were studied there. Lim \cite{LL} independently introduced real tensor eigenvalues and obtained some existence results using a variational approach.
\begin{definition} \label{d23} 
	Let $\mathcal{A}  = (a_{i_1\cdots i_m}) \in T_{m, n}$.   \begin{itemize}
		\item[(i)] A number $\lambda \in \mathbb{C}$ is called an {\bf eigenvalue} of $\mathcal{A}$ iff there is a nonzero vector $\x \in \mathbb{C}^n$ such that
		\begin{equation} \label{eig}
		\mathcal{A} \x^{m-1} = \lambda \x^{[m-1]},
		\end{equation}
		and $\x$  is called an {\bf eigenvector} of $\mathcal{A}$, associated with $\lambda$.  An eigenvalue $\lambda$ corresponding a real eigenvector $\x$ is real and is called an {\bf
			$H$-eigenvalue}, and $\x$ is called an {\bf $H$-eigenvector} of $\mathcal{A}$, respectively;
		\item[(ii)]  A number $\lambda \in \mathbb{C}$ is called an {\bf E-eigenvalue} of $\mathcal{A}$ iff there is a nonzero vector $\x \in \mathbb{C}^n$ such that
		\begin{equation} \label{eig1}
		\mathcal{A} \x^{m-1} = \lambda \x,  \ \ \x^\top \x = 1,
		\end{equation}
		and $\x$ is
		called an {\bf $E$-eigenvector} of $\mathcal{A}$,  associated with $\lambda$.  An $E$-eigenvalue  $\lambda$ corresponding a real $E$-eigenvector $\x$ is real and is called an {\bf
			$Z$-eigenvalue}, and $\x$ is called an {\bf $Z$-eigenvector} of $\mathcal{A}$, respectively.\end{itemize}
\end{definition}
Recently, Qi \cite{LQ4} introduced and used the following concepts
for studying the properties of hypergraphs.
\begin{definition} \label{d24}  A real number
	$\lambda$ is said to be
	\begin{itemize}
		\item[(i)]  an {\bf
			$H^+$-eigenvalue of $\mathcal{A}$}  iff it is an $H$-eigenvalue and its $H$-eigenvector $\x\in
		\mathbb{R}^n_+$;
		\item[(ii)] an {\bf  $H^{++}$-eigenvalue of $\mathcal{A}$},
		iff it is an $H$-eigenvalue and its $H$-eigenvector $\x\in \mathbb{R}^n_{++}$.
		\item[(iii)] a {\bf $Z^+$-eigenvalue of $\mathcal{A}$} (\cite{SQ-15}) iff it is a $Z$-eigenvalue and its $Z$-eigenvector $\x\in \mathbb{R}^n_+$;
		\item[(iv)] a {\bf  $Z^{++}$-eigenvalue of $\mathcal{A}$} (\cite{SQ-15}) iff it is a $Z$-eigenvalue and its $Z$-eigenvector $\x\in \mathbb{R}^n_{++}$.
	\end{itemize}
\end{definition}

Recall that an operator $T : \mathbb{R}^n \to \mathbb{R}^n$ is called {\bf
	positively homogeneous} iff $T(t\x)=tT(\x)$ for each $t>0$ and all
$\x\in \mathbb{R}^n$.  For $\x \in \mathbb{R}^n$, it is known well that
$$\|\x\|_\infty:=\max\{|x_i|;i \in I_n \}\mbox{ and } \|\x
\|_p:=\left(\sum_{i=1}^n|x_i|^p\right)^{\frac1p}  \ (p\geq1)$$ are two main norms
defined on $\mathbb{R}^n$.  Then for a continuous, positively homogeneous
operator $T : \mathbb{R}^n\to \mathbb{R}^n$,  it is obvious that
\begin{equation} \label{eq23}\|T\|_p:=\max_{\|\x\|_p=1}\|T(\x)\|_p\mbox{ and }\|T\|_\infty:=\max_{\|\x\|_\infty=1}\|T(\x)\|_\infty\end{equation} are two
operator norms of $T$.

Let $\mathcal{A} \in T_{m, n}$.  Define an operator $T_\mathcal{A} : \mathbb{R}^n \to \mathbb{R}^n$
by for any $\x \in \mathbb{R}^n$,
\begin{equation} \label{TA}
T_\mathcal{A} (\x) := \begin{cases}\|\x \|_2^{2-m}\mathcal{A} \x^{m-1},\ \x \neq\0\\
\0,\ \ \ \ \ \ \ \ \ \ \  \ \ \ \ \ \ \
\x =\0.\end{cases}
\end{equation}
When $m$ is even, define another operator $F_\mathcal{A} : \mathbb{R}^n \to \mathbb{R}^n$ by for any $\x \in \mathbb{R}^n$,
\begin{equation} \label{FA}
F_\mathcal{A} (\x) :=\left(\mathcal{A}  \x^{m-1}\right)^{\left[\frac1{m-1}\right]}.
\end{equation}
Clearly, both
$F_\mathcal{A}$ and $T_\mathcal{A}$ are continuous and positively homogeneous. The
following upper bounds and properities of the operator norm were established  by
Song and Qi \cite{SQ}.
\begin{lemma} \label{le:23}   (Song and Qi \cite[Theorem 4.3,Lemma 2.1]{SQ})  Let $\mathcal{A} = (a_{i_1\cdots i_m}) \in T_{m, n}$. Then
	\begin{itemize}
		\item[(i)] $\|F_\mathcal{A}(\x)\|_{\infty}\leq \|F_\mathcal{A}\|_{\infty}\|\x\|_{\infty}$ and $\|F_\mathcal{A}(\x)\|_p\leq \|F_\mathcal{A}\|_p\|\x\|_p$;
		\item[(ii)]$\|T_\mathcal{A}(\x)\|_{\infty}\leq \|T_\mathcal{A}\|_{\infty}\|\x\|_{\infty}$ and $\|T_\mathcal{A}(\x)\|_p\leq \|T_\mathcal{A}\|_p\|\x\|_p$;
		\item[(iii)] $\|T_\mathcal{A}\|_{\infty}\leq \max\limits_{i\in I_n}\left(\sum\limits_{i_2,\cdots,i_m=1}^{n}|a_{ii_2\cdots i_m}|\right)$;
		\item[(iv)] $\|F_\mathcal{A}\|_{\infty}\leq\max\limits_{i\in I_n}\left(\sum\limits_{i_2,\cdots,i_m=1}^{n}|a_{ii_2\cdots i_m}|\right)^{\frac{1}{m-1}}$, when $m$ is even.
\end{itemize} \end{lemma}

\begin{lemma}\label{le:24}
	Let $\mathcal{A} = (a_{i_1\cdots i_m}) \in T_{m, n}$. Then
	\begin{itemize}
		\item[(i)] $\frac1{\sqrt[p]{n}}\|F_\mathcal{A}\|_{\infty}\leq\|F_\mathcal{A}\|_p\leq \sqrt[p]{n}\|F_\mathcal{A}\|_{\infty}$;
		\item[(ii)]$\frac1{\sqrt[p]{n}}\|T_\mathcal{A}\|_{\infty}\leq\|T_\mathcal{A}\|_p\leq \sqrt[p]{n}\|T_\mathcal{A}\|_{\infty}$;
		\item[(iii)] $\|F_\mathcal{A}\|_p\leq \left(\sum\limits_{i=1}^{n}\left(\sum\limits_{i_2,\cdots,i_m=1}^{n}|a_{ii_2\cdots i_m}|\right)^{\frac{p}{m-1}}\right)^{\frac1p}$ when $m$ is even;
		\item[(iv)] $\|T_\mathcal{A}\|_p\leq n^{\frac{m-2}p}\left(\sum\limits_{i=1}^{n}\left(\sum\limits_{i_2,\cdots,i_m=1}^{n}|a_{ii_2\cdots i_m}|\right)^p\right)^{\frac1p}$.
	\end{itemize} 	
\end{lemma}
\begin{proof}
	(i)  It follows from the definitions (\ref{eq23}) of the operator norm and the fact that $\|\x\|_\infty\leq\|\x\|_p\leq\sqrt[p]{n}\|\x\|_\infty$ that $$\begin{aligned}
	\|F_\mathcal{A}\|_p=\max_{\|\x\|_p=1}\|F_\mathcal{A}(\x)\|_p\leq& \max_{\|\x\|_p=1}\sqrt[p]{n}\|F_\mathcal{A}(\x)\|_\infty\\\leq&\max_{\|\x\|_p=1}\sqrt[p]{n}\|F_\mathcal{A}\|_\infty\|\x\|_\infty\\\leq&\max_{\|\x\|_p=1}\sqrt[p]{n}\|F_\mathcal{A}\|_\infty\|\x\|_p\\=&\sqrt[p]{n}\|F_\mathcal{A}\|_\infty
	\end{aligned}$$
	and 	
	$$\begin{aligned}
	\|F_\mathcal{A}\|_\infty=\max_{\|\x\|_\infty=1}\|F_\mathcal{A}(\x)\|_\infty\leq& \max_{\|\x\|_\infty=1}\|F_\mathcal{A}(\x)\|_p\\\leq&\max_{\|\x\|_\infty=1}\|F_\mathcal{A}\|_p\|\x\|_p\\
	\leq&\max_{\|\x\|_\infty=1}\sqrt[p]{n}\|F_\mathcal{A}\|_p\|\x\|_\infty\\=&\sqrt[p]{n}\|F_\mathcal{A}\|_p.
	\end{aligned}$$
	This proves (i). Similarly, (ii) is easy to obtain.
	
	(iii) It follows from the definition (\ref{eq23}) of the operator norm that
	$$\begin{aligned}
	\|F_\mathcal{A}\|_p^p=&(\max_{\|x\|_p=1}\|F_\mathcal{A}x\|_p)^p=\max_{\|x\|_p=1}\|F_\mathcal{A}x\|_p^p\\
	=& \max_{\|x\|_p=1}\sum_{i=1}^{n}\left|\left(\sum\limits_{i_2,\cdots,i_m=1}^{n}a_{ii_2\cdots i_m}x_{i_2}x_{i_3}\cdots x_{i_m}\right)^{\frac{1}{m-1}}\right|^p\\
	\leq&\max_{\|x\|_p=1}\sum_{i=1}^{n}\left(\sum\limits_{i_2,\cdots,i_m=1}^{n}|a_{ii_2\cdots i_m}||x_{i_2}||x_{i_3}|\cdots |x_{i_m}|\right)^{\frac{p}{m-1}}\\
	\leq&\max_{\|x\|_p=1}\sum_{i=1}^{n}\left(\sum\limits_{i_2,\cdots,i_m=1}^{n}|a_{ii_2\cdots i_m}|\|x\|_p^{m-1}\right)^{\frac{p}{m-1}}\\
	=&\max_{\|x\|_p=1}\sum_{i=1}^{n}\left(\sum\limits_{i_2,\cdots,i_m=1}^{n}|a_{ii_2\cdots i_m}|\right)^{\frac{p}{m-1}}\|x\|_p^p\\
	=&\sum_{i=1}^{n}\left(\sum\limits_{i_2,\cdots,i_m=1}^{n}|a_{ii_2\cdots i_m}|\right)^{\frac{p}{m-1}}.
	\end{aligned}$$
	
	(iv)
	It follows from the definition of the norm that $\|x\|_2\geq \frac1{\sqrt[p]{n}}\|x\|_p$ and
	$$\begin{aligned} \|T_\mathcal{A}\|_p^p=&\max_{\|x\|_\infty=1}\|T_\mathcal{A}x\|_p^p\\
	=& \max_{\|x\|_p=1}\sum_{i=1}^{n}\left|\|x\|_2^{-(m-2)}\sum\limits_{i_2,\cdots,i_m=1}^{n}a_{ii_2\cdots i_m}x_{i_2}x_{i_3}\cdots x_{i_m}\right|^p\\
	\leq&\max_{\|x\|_p=1}\|x\|_2^{-p(m-2)}\sum_{i=1}^{n}\left(\sum\limits_{i_2,\cdots,i_m=1}^{n}|a_{ii_2\cdots i_m}||x_{i_2}||x_{i_3}|\cdots |x_{i_m}|\right)^p\\
	\leq&\max_{\|x\|_p=1}n^{m-2}\|x\|_p^{-p(m-2)}\|x\|_p^{p(m-1)}\sum_{i=1}^{n}\left(\sum\limits_{i_2,\cdots,i_m=1}^{n}|a_{ii_2\cdots i_m}|\right)^p \\
	=&n^{m-2}\sum_{i=1}^{n}\left(\sum\limits_{i_2,\cdots,i_m=1}^{n}|a_{ii_2\cdots i_m}|\right)^p.
	\end{aligned}$$ This completes the proof.
\end{proof}
The following conclusion about the solution to TCP$(\mathcal{A},\q)$ with P-tensor $\mathcal{A}$ is obtained by Song and Qi \cite{SQ2015,SQ15}
\begin{lemma}\label{le:25} (Song and Qi \cite[Corollary 3.3, Theorem 3.4]{SQ2015} and \cite[Theorem 3.2]{SQ15}) Let $\mathcal{A} \in T_{m, n}$ be a  strictly semi-positive tensor. Then the TCP$(\mathcal{A},\q)$ has a solution for all $\q\in \mathbb{R}^n$, and  has only zero vector solution for $\q\geq\0$.
\end{lemma}

\section{Properties of  Semi-positive Tensors}

Recently, Song and Yu \cite{SY15} defined a quantity for a strictly semi-positive tensor $\mathcal{A}$.
\begin{equation} \label{Tbeta}
\beta(\mathcal{A}):=\min_{\|\x \|_\infty=1\atop \x\geq\0}\max_{i \in I_n}x_i(\mathcal{A}\x^{m-1})_i.
\end{equation}

\subsection{Monotonicity and boundedness of $\beta(\mathcal{A})$}

We now establish the monotonicity and boundedness of the constant  $\beta(\mathcal{A})$ for a (strictly) semi-positive tensor.   The proof technique is similar to the proof technique of \cite[Theorem 4.3]{SQ-2015} and \cite[Theorem 1.2]{XZ}.   For completeness, we give the proof here.
\begin{theorem}\label{t31}  Let  $\mathcal{D} =\mbox{diag}(d_1, d_2,\cdots, d_n)$ be a nonnegative diagonal tensor in $T_{m, n}$ and $\mathcal{A} = (a_{i_1\cdots i_m})$ be a semi-positive tensor in $T_{m, n}$. Then
	\begin{itemize}
		\item[(i)] $\beta(\mathcal{A})\leq\beta(\mathcal{A}+\mathcal{D})$;
		\item[(ii)] $\beta(\mathcal{A})\leq\beta(\mathcal{A}^J_r)$ for all principal sub-tensors $\mathcal{A}^J_r$;
		\item[(iii)] $\beta(\mathcal{A})\leq n^{\frac{m-2}2}\|T_{\mathcal{A}}\|_\infty$;
		\item[(iv)]  $\beta(\mathcal{A})\leq\|F_{\mathcal{A}}\|^{m-1}_\infty$   when $m$ is even.
	\end{itemize}
\end{theorem}
\begin{proof}
	(i) By the definition of semi-positive tensors, clearly $\mathcal{A}+\mathcal{D}$ is a semi-positive tensor.  Then $\beta(\mathcal{A}+\mathcal{D})$ is well-defined. Then we have
	$$\begin{aligned}
	\beta(\mathcal{A})=&\min_{\|\x\|_\infty=1\atop \x\geq\0}\max_{i\in I_n}x_i(\mathcal{A}\x^{m-1})_i\\
	\leq &\min_{\|\x\|_\infty=1\atop \x\geq\0}\max_{i\in I_n}\left(x_i(\mathcal{A} \x^{m-1})_i+d_ix_i^m\right)\\
	=& \min_{\|\x\|_\infty=1\atop \x\geq\0}\max_{i\in I_n}x_i\left((\mathcal{A}+\mathcal{D})\x^{m-1}\right)_i\\
	=&\beta({\mathcal{A}+\mathcal{D}}).
	\end{aligned}$$
	
	(ii) Let a principal sub-tensor $\mathcal{A}^J_r$ of $\mathcal{A}$ be given. Then for each nonzero vector $\z=(z_1,\cdots,z_r)^\top \in \mathbb{R}^r_+$, we may define $\y(\z)=(y_1(\z),y_2(\z), \cdots, y_n(\z))^\top \in \mathbb{R}^n_+$ with $y_i(\z)=z_i $ for $ i\in J$ and $y_i(\z)=0$ for $i\notin J$.  Thus $\|\z\|_\infty=\|\y(\z)\|_\infty$, and hence,
	$$\begin{aligned}
	\beta(\mathcal{A})=&\min_{\|\x\|_\infty=1\atop \x\geq\0}\max_{i\in I_n}x_i(\mathcal{A}\x^{m-1})_i\\
	\leq & \min_{\|\y(\z)\|_\infty=1\atop \y(\z)\geq\0}\max_{i\in I_n}(\y(\z))_i(\mathcal{A}(\y(\z))^{m-1})_i\\
	=& \min_{\|\z\|_\infty=1\atop \z\geq\0}\max_{i\in I_n}z_i(\mathcal{A}^J_r\z^{m-1})_i\\
	=&\beta({\mathcal{A}^J_r}).
	\end{aligned}$$
	
	(iii) It follows from Lemma \ref{le:23} that for each nonzero vector $\x=(x_1,\cdots,x_n)^\top \in \mathbb{R}^n_+$ and each $i\in I_n$,
	$$x_i(\mathcal{A}\x^{m-1})_i=x_i(\|\x\|^{m-2}_2T_\mathcal{A}\x^{m-1})_i \leq \|\x\|^{m-2}_2\|\x\|_\infty \|T_\mathcal{A}(\x)\|_\infty\leq\|\x\|^{m-2}_2\|T_\mathcal{A}\|_\infty\|\x\|_\infty^2,$$
	Then using $\|\x\|_2\leq \sqrt{n}\|\x\|_\infty$, we have $$\max_{i\in I_n}x_i(\mathcal{A}\x^{m-1})_i\leq\|\x\|^{m-2}_2\|T_\mathcal{A}\|_\infty\|\x\|^2_\infty\leq n^{\frac{m-2}2}\|T_\mathcal{A}\|_\infty\|\x\|^m_\infty.$$
	Therefore, we have $$\beta(\mathcal{A})=\min_{\|\x\|_\infty=1\atop \x\geq\0}\max_{i\in I_n}x_i(\mathcal{A}\x^{m-1})_i\leq n^{\frac{m-2}2} \|T_\mathcal{A}\|_\infty.$$
	
	(iv) It follows from Lemma \ref{le:23} that for each nonzero vector $\x=(x_1,\cdots,x_n)^\top \in \mathbb{R}^n_+$ and each $i\in I_n$,
	$$x_i(\mathcal{A}\x^{m-1})_i=x_i(F_\mathcal{A}\x^{m-1})^{m-1}_i \leq \|\x\|_\infty \|F_\mathcal{A}(\x)\|^{m-1}_\infty\leq\|F_\mathcal{A}\|_\infty^{m-1}\|\x\|_\infty^m,$$
	Then  we have $$\max_{i\in I_n}x_i(\mathcal{A}\x^{m-1})_i\leq\|F_\mathcal{A}\|_\infty^{m-1}\|\x\|^m_\infty.$$
	Therefore, we have $$\beta(\mathcal{A})=\min_{\|\x\|_\infty=1\atop \x\geq\0}\max_{i\in I_n}x_i(\mathcal{A}\x^{m-1})_i\leq  \|F_\mathcal{A}\|_\infty^{m-1}.$$ The desired conclusions follow.
\end{proof}

\subsection{Necessary  and sufficient conditions of strictly semi-positive tensor}
We now give necessary  and sufficient conditions for a tensor $A \in T_{m, n}$ to be a strictly semi-positive tensor, based upon the constant $\beta(\mathcal{A})$.

\begin{theorem}\label{t32}  Let $\mathcal{A} \in T_{m, n}$.  Then
	\begin{itemize}
		\item[(i)] $\mathcal{A}$ is a strictly semi-positive tensor if and only if $\beta(\mathcal{A})>0$;
		\item[(ii)]$\beta(\mathcal{A})\geq0$ if $\mathcal{A}$ is a semi-positive tensor.
	\end{itemize}
\end{theorem}
\begin{proof}
	(i) Let $\mathcal{A}$ be strictly semi-positive. Then it follows from the definition of strictly semi-positive tensors that for each  $\x=(x_1,x_2,\cdots,x_n)^\top \in \mathbb{R}^n_+$ and $\x\ne\0$, there exists $k\in I_n$ such that \begin{equation}
	\label{eq32}x_k>0\mbox{ and }(\mathcal{A} \x^{m-1})_k>0, \mbox{ i.e., }x_k(\mathcal{A} \x^{m-1})_k>0.
	\end{equation} So, we have $$\max_{i\in I_n}x_i(\mathcal{A} \x^{m-1})_i>0.$$ Thus $$\beta(\mathcal{A})=\min_{\|\x \|_\infty=1\atop \x\geq\0}\max_{i\in I_n}x_i(\mathcal{A}\x^{m-1})_i>0.$$
	
	If $\beta(\mathcal{A})>0$, then it is obvious that for each $\y \in \mathbb{R}^n_+$ and  $\y\ne\0$,
	$$\max_{i\in I_n}\left(\frac{\y}{\|\y\|_\infty}\right)_i\left(\mathcal{A}\left(\frac{\y}{\|\y\|_\infty}\right)^{m-1}\right)_i\geq\beta(\mathcal{A}) >0.$$ Hence, by $\|\y\|_\infty>0$, we have $$\max_{i\in I_n}y_i(\mathcal{A}\y^{m-1})_i>0.$$
	Thus   $y_k(\mathcal{A}\y^{m-1})_k>0$ for some $k\in I_n$, i.e., $\mathcal{A}$ is a strictly semi-positive tensor.
	
	(ii)  The proof is similar to ones of (i), we omit it.
\end{proof}
\begin{corollary}\label{c33}  Let $\mathcal{A} \in T_{m, n}$.  Then
	\begin{itemize}
		\item[(i)] $\min\limits_{i\in I_n}a_{ii\cdots i}\geq\beta(\mathcal{A})>0$ if $\mathcal{A}$ is a strictly semi-positive tensor;
		\item[(ii)] $\min\limits_{i\in I_n}a_{ii\cdots i}\geq\beta(\mathcal{A})\geq0$ if $\mathcal{A}$ is a semi-positive tensor.
	\end{itemize}	
\end{corollary}
\begin{proof}
	(i) It follows from Theorem \ref{t31} (ii) that  $$\beta(\mathcal{A})\leq\beta(\mathcal{A}^J_r)\mbox{ for all }J\subset I_n, r\in I_n.$$
	Choose $J=\{i\}$ for each $i\in I_n$ and $r=1$. Then for all $\x=x_i\in\mathbb{R}^1$, we have $$\mathcal{A}^J_r\x^{m-1}=a_{ii\cdots i}x_i^{m-1}.$$ Let $\|\x\|_\infty=1$. Then $\x=1$ or $-1$, and hence
	$$\beta(\mathcal{A}^J_r)=\min_{\|\x \|_\infty=1\atop \x\geq\0}\max_{i\in J}x_i(\mathcal{A}^J_r\x^{m-1})_i=1\times a_{ii\cdots i} \times 1^{m-1}=a_{ii\cdots i}.$$ Since $i\in I_n$ is arbitrary, the desired conclusion follows.
	
	Similarly, (ii) is easy to obtain, we omit it.
\end{proof}
Combing the above conclusions and Lemma \ref{le:21}, the following are easy to obtain.
\begin{corollary}\label{c34}  Let $\mathcal{A} \in S_{m, n}$.  Then
	\begin{itemize}
		\item[(i)] $\mathcal{A}$ is a strictly copositive tensor if and only if $\beta(\mathcal{A})>0$;
		\item[(ii)]$\beta(\mathcal{A})\geq0$ if $\mathcal{A}$ is a copositive tensor.
		\item[(iii)] $\min\limits_{i\in I_n}a_{ii\cdots i}\geq\beta(\mathcal{A})>0$ if $\mathcal{A}$ is strictly copositive;
		\item[(iv)] $\min\limits_{i\in I_n}a_{ii\cdots i}\geq\beta(\mathcal{A})\geq0$ if $\mathcal{A}$ is copositive.
	\end{itemize}	
\end{corollary}
\subsection{Eigenvalues of a semi-positive tensor}
\begin{theorem}\label{t34} Let $\mathcal{A} \in T_{m, n}$.
	\begin{itemize}
		\item[(i)] If $\mathcal{A}$ is a strictly semi-positive tensor, then all $H^+$-eigenvalues of $\mathcal{A}$ are positive;
		\item[(ii)] If $\mathcal{A}$ is a semi-positive tensor, then all $H^+$-eigenvalues of $\mathcal{A}$ are nonnegative;
		\item[(iii)] If $\mathcal{A}$ is a strictly semi-positive tensor, then all $Z^+$-eigenvalues of $\mathcal{A}$ are positive;
		\item[(iv)] If $\mathcal{A}$ is a semi-positive tensor, then all $Z^+$-eigenvalues of $\mathcal{A}$ are nonnegative.
	\end{itemize}
\end{theorem}
\begin{proof}
	(i) Let $\mathcal{A}$ be a strictly semi-positive tensor. Then it follows from the definition of strictly semi-positive tensors that for each  $\x=(x_1,x_2,\cdots,x_n)^\top \in \mathbb{R}^n_+$ and $\x\ne\0$, there exists $k\in I_n$ such that \begin{equation} \label{eq33} x_k>0\mbox{ and }(\mathcal{A} \x^{m-1})_k>0. \end{equation}
	Let $\lambda$ be an H$^+$-eigenvalue of $\mathcal{A}$. Then there exists a vector $\y\in\mathbb{R}^n_+$ and $\y\ne\0$ such that
	\begin{equation} \label{eq34}\left(\mathcal{A}\y^{m-1}\right)_i=\lambda y_i^{m-1}\mbox{ for all }i\in I_n. \end{equation}
	Putting $\x=\y$ in (\ref{eq33}), there exists $k\in I_n$ such that
	\begin{equation} \label{eq35}y_k>0\mbox{ and }(\mathcal{A} \y^{m-1})_k>0.\end{equation}
	Combining (\ref{eq34}) and (\ref{eq35}), we have
	$$0<\left(\mathcal{A}\y^{m-1}\right)_k=\lambda y_k^{m-1},$$
	and so, $\lambda>0$.  Since $\lambda$ is an arbitrary $H^+$-eigenvalue of $\mathcal{A}$, the desired conclusion follows.
	
	Similarly, it is easy to prove (ii), (iii) and (iv).	
\end{proof}

By Lemma \ref{le:22}, the following corollary is easy to be proved.
\begin{corollary} \label{c35}  Let $\mathcal{A} \in T_{m, n}$. Then
	\begin{itemize}
		\item[(i)]  all $H^+$-eigenvalues of each principal sub-tensor of $\mathcal{A}$ are nonnegative (positive) if $\mathcal{A}$ is a (strictly) semi-positive tensor;
		\item[(ii)]  all $Z^+$-eigenvalues of each principal sub-tensor of $\mathcal{A}$ are nonnegative (positive) if $\mathcal{A}$ is a (strictly) semi-positive tensor.
	\end{itemize}
\end{corollary}
By Lemma \ref{le:21}, the following corollary is easy to be proved.
\begin{corollary} \label{c36}  Let $\mathcal{A} \in S_{m, n}$. Then
	\begin{itemize}
		\item[(i)]  all $H^+$-eigenvalues of each principal sub-tensor of $\mathcal{A}$ are nonnegative (positive) if $\mathcal{A}$ is a (strictly) copositive tensor;
		\item[(ii)] all $Z^+$-eigenvalues of each principal sub-tensor of $\mathcal{A}$ are nonnegative (positive) if $\mathcal{A}$ is a (strictly) copositive tensor.
	\end{itemize}
\end{corollary}
\subsection{Upper bounds of $\beta(\mathcal{A})$}
The quantity $\beta(\mathcal{A})$ is in general not easy to compute. However, it is easy to derive some upper bounds for them when $\mathcal{A}$ is a strictly semi-positive tensor.  Next we will establish some smaller upper bounds. For this purpose, we introduce two quantities about a strictly semi-positive tensor $\mathcal{A}$:
\begin{equation}\label{eq:36}
\delta_{H^+}(\mathcal{A}) := \min\{\lambda_{H^+}(\mathcal{A}_r^J); J\subset I_n, r\in I_n\},
\end{equation}
where  $\lambda_{H^+}(\mathcal{A})$ denotes the smallest of $H^+$-eigenvalues (if any exists) of a strictly semi-positive tensor $\mathcal{A}$;
\begin{equation}\label{eq:37}
\delta_{Z^+}(\mathcal{A}) := \min\{\lambda_{Z^+}(\mathcal{A}_r^J); J\subset I_n, r\in I_n\},
\end{equation}
where  $\lambda_{Z^+}(\mathcal{A})$ denotes the smallest $Z^+$-eigenvalue (if any exists) of a strictly semi-positive tensor $\mathcal{A}$. The above two minimums range over those principal sub-tensors of $\mathcal{A}$ which indeed have $H^+$-eigenvalues ($Z^+$-eigenvalues).

It follows from Lemma \ref{le:22} (or Corollary \ref{c33}) and Corollary \ref{c35} that all principal diagonal entries of $\mathcal{A}$ are positive and all  $H^+(Z^+)$-eigenvalues of each principal sub-tensor of $\mathcal{A}$ are positive when $m$ is even. So $\delta_{H^+}(\mathcal{A})$ and $\delta_{Z^+}(\mathcal{A})$ are well defined, finite and positive.  Now we give some upper bounds of $\beta({\mathcal{A}})$ using the quantities $\delta_{H^+}(\mathcal{A})$ and $\delta_{Z^+}(\mathcal{A})$.

\begin{proposition}\label{p36}  Let $\mathcal{A} \in T_{m, n}$ ($m\geq2$) be a strictly semi-positive tensor.  Then
	\begin{equation}\label{eq:38}
	\delta_{H^+}(\mathcal{A})\leq \min_{i\in I_n}a_{ii\cdots i}\mbox{ and } \delta_{Z^+}(\mathcal{A})\leq \min_{i\in I_n}a_{ii\cdots i}.
	\end{equation}
\end{proposition}
\begin{proof} It follows from Lemma \ref{le:22} (or Corollary \ref{c33}) that $a_{ii\cdots i}>0\mbox{ for all }i\in I_n.$ Since $\mathcal{A}^J_1=(a_{ii\cdots i})$ ($J=\{i\}$) is $m$-order 1-dimensional principal sub-tensor of $\mathcal{A}$,  $a_{ii\cdots i}$ is a H$^+$-eigenvalue of $\mathcal{A}^J_1$ for all $i\in I_n$, and hence
	$$\delta_{H^+}(\mathcal{A})\leq \min_{i\in I_n}a_{ii\cdots i}.$$
	Similarly, it is easy to prove the other inequality.	\end{proof}
\begin{theorem}\label{t36}  Let $\mathcal{A} \in T_{m, n}$ ($m\geq2$) be a strictly semi-positive tensor.  Then
	\begin{itemize}
		\item[(i)] $\beta({\mathcal{A}})\leq \delta_{H^+}(\mathcal{A});$
		\item[(ii)] $\beta({\mathcal{A}})\leq n^{\frac{m-2}2}\delta_{Z
			^+}(\mathcal{A})$ if $m$ is even.
	\end{itemize}
\end{theorem}
\begin{proof}
	(i) Let $\delta=\delta_{H^+}(\mathcal{A})$ and $\mathcal{B}=\mathcal{A}-\delta\mathcal{I}$, where $\mathcal{I}$ is unit tensor.  Then it follows from the definition of $\delta_{H^+}(\mathcal{A})$ that $\delta$ is  an $H^+$-eigenvalue of a principal sub-tensor $\mathcal{A}^J_r$ of $\mathcal{A}$. Then there exists $\x_J\in \mathbb{R}^r_+\setminus\{\0\}$ such that  $$\left(\mathcal{A}^J_r-\delta\mathcal{I}^J_r\right)(\x_J)^{m-1}=\mathcal{A}^J_r(\x_J)^{m-1}-\delta(\x_J)^{[m-1]}=\0.$$
	So the principal sub-tensor $\mathcal{B}^J_r=\mathcal{A}^J_r-\delta\mathcal{I}^J_r$ of $\mathcal{B}$ is not a strictly semi-positive tensor. Thus  it follows from Lemma \ref{le:22}  that $\mathcal{B}=\mathcal{A}-\delta\mathcal{I}$ is not  strictly semi-positive. From Theorem \ref{t32} (i), it follows that $\beta(\mathcal{B})\leq0,$ and hence, by the definition of $\beta(\mathcal{B})$,
	there exists a vector $\y$ with $\|\y\|_\infty=1$ such that
	$$\max_{i\in I_n}y_i(\mathcal{B}\y^{m-1})_i=\max_{i\in I_n}\left(y_i(\mathcal{A}\y^{m-1})_i-\delta y_i^{m}\right)\leq 0.$$
	So, we have $$y_k(\mathcal{A}\y^{m-1})_k-\delta y_k^{m}\leq\max_{i\in I_n}\left(y_i(\mathcal{A}\y^{m-1})_i-\delta y_i^{m}\right)\leq 0 \mbox{ for all }k\in I_n,$$
	and so, $y_k(\mathcal{A}\y^{m-1})_k\leq\delta y_k^{m} \mbox{ for all }k\in I_n.$ Thus,
	$$\max_{i\in I_n}y_i(\mathcal{A}\y^{m-1})_i\leq\delta\max_{i\in I_n} y_i^{m}\leq \delta \|\y\|_\infty^m=\delta.$$
	This implies that $\beta({\mathcal{A}})\leq \delta_{H^+}(\mathcal{A})$.
	
	(ii)  Let $\gamma=\delta_{Z^+}(\mathcal{A})$ and $\mathcal{B}=\mathcal{A}-\gamma\mathcal{E}$, where $\mathcal{E}=I^{\frac{m}2}_2$ and $I_2$ is $n\times n$ unit matrix ($\mathcal{E}\x^{m-1}=\|\x\|^{m-2}_2\x$, see Chang, Pearson, Zhang \cite{CPZ09}).  Then $\gamma$ is  a Z$^+$-eigenvalue of a principal sub-tensor $\mathcal{A}^J_r$ of $\mathcal{A}$. Then there exists $\z_J\in \mathbb{R}^r_+\setminus\{\0\}$ such that $(\z_J)^\top\z_J=1$ and $$\left(\mathcal{A}^J_r-\gamma\mathcal{E}^J_r\right)(\z_J)^{m-1}=\mathcal{A}^J_r(\z_J)^{m-1}-\gamma\z_J=\0.$$ So the principal sub-tensor $\mathcal{B}^J_r=\mathcal{A}^J_r-\gamma\mathcal{E}^J_r$ of $\mathcal{B}$ is not strictly semi-positive. Thus it follows from Lemma \ref{le:22} that $\mathcal{B}=\mathcal{A}-\gamma\mathcal{E}$ is not strictly semi-positive. From Theorem \ref{t32} (i), it follows that $\beta(\mathcal{B})\leq0.$
	So, there exists a vector $\y$ with $\|\y\|_\infty=1$ such that
	$$\max_{i\in I_n}y_i(\mathcal{B}\y^{m-1})_i=\max_{i\in I_n}\left(y_i(\mathcal{A}\y^{m-1})_i-\gamma\|\y\|_2^{m-2} y_i^2\right)\leq 0.$$
	Thus, using the similar proof technique of (i) and $\|\y\|_2\leq \sqrt{n}\|\y\|_\infty$, we have
	$$\max_{i\in I_n}y_i(\mathcal{A}\y^{m-1})_i\leq\gamma \|\y\|^{m-2}_2 \|\y\|_\infty^2=n^{\frac{m-2}2}\gamma\|\y\|_\infty^m=n^{\frac{m-2}2}\gamma.$$
	The desired inequality follows.
\end{proof}
From Lemma \ref{le:21} and Theorem \ref{t36}, the following corollary  follows.
\begin{corollary}\label{c310}  Let $\mathcal{A} \in S_{m, n}$ ($m\geq2$) be strictly copositive.  Then
	\begin{itemize}
		\item[(i)] $\beta({\mathcal{A}})\leq \delta_{H^+}(\mathcal{A})\leq \min_{i\in I_n}a_{ii\cdots i};$
		\item[(ii)] $\beta({\mathcal{A}})\leq n^{\frac{m-2}2}\delta_{Z
			^+}(\mathcal{A})$ if $m$ is even.
	\end{itemize}
\end{corollary}

\begin{question}\label{Q1}  It is known from  Theorem \ref{t36} that for  a strictly semi-positive tensor $\mathcal{A}$,
	$$\min_{i\in I_n}a_{ii\cdots i}\geq \delta_{H^+}(\mathcal{A})\geq\beta({\mathcal{A}})>0.$$
	Then we have the following questions for further research. \begin{itemize}
		\item[(i)]  Does the constant $\beta({\mathcal{A}})$ have a strictly positive lower bound?
		\item[(ii)]  Are the above upper bounds the best?
	\end{itemize}
\end{question}

\subsection{Boundedness of solution set of TCP$(\mathcal{A},\q)$}

Song and Qi \cite{SQ13} introduced the concept of Pareto $H(Z)$-eigenvalues and used it to portray the (strictly) copositive tensor. The number and computation of Pareto $H(Z)$-eigenvalue see Ling, He and Qi \cite{LHQ2015,LHQ2016}, Chen, Yang and Ye \cite{CYY}.

\begin{definition} \label{d31}
	Let $\mathcal{A}  = (a_{i_1\cdots i_m}) \in T_{m, n}$.   A real number $\mu$ is said to be
	\begin{itemize}
		\item[(i)]a {\bf Pareto $H$-eigenvalue} of $\mathcal{A}$ iff there is a non-zero vector $\x\in \mathbb{R}^n$ satisfying \begin{equation}\label{eq:41} \mathcal{A}\x^m=\mu \x^\top\x^{[m-1]}, \
		\mathcal{A}\x^{m-1}-\mu \x^{[m-1]} \geq0,\ \x\geq 0,\end{equation}
		where $\x^{[m-1]}=(x_1^{m-1},x_2^{m-1},\cdots,x_n^{m-1})^\top$.
		\item[(ii)] a {\bf Pareto $Z$-eigenvalue} of $\mathcal{A}$ iff there is a non-zero vector $\x\in \mathbb{R}^n$ satisfying \begin{equation}\label{eq:42}  \mathcal{A}\x^m=\mu (\x^\top\x)^{\frac{m}2}, \
		\mathcal{A}\x^{m-1}-\mu (\x^\top\x)^{\frac{m}2-1}\x \geq0,\  \x\geq 0. \end{equation}
	\end{itemize}
\end{definition}

Let $$\lambda(\mathcal{A})=\min\{\lambda; \lambda \mbox{ is Pareto H-eigenvalue of  }\mathcal{A}\}$$ and
$$\mu(\mathcal{A})=\min\{\mu; \mu \mbox{ is Pareto Z-eigenvalue of  }\mathcal{A}\}.$$ Song and Yu \cite{SY15} obtained the following upper bounds of solution set of TCP$(\mathcal{A},\q)$.

\begin{lemma}\label{le:31}  (Song and Yu \cite[Theorems 3.3,3.4,3.5]{SY15}) Let $\mathcal{A}  = (a_{i_1\cdots i_m}) \in T_{m,n}$ be strictly semi-positive and let $\x$ be a solution of the {\em TCP}$(\mathcal{A}, \q)$.  Then\begin{itemize}
		\item[(i)]
		$\|\x\|_\infty^{m-1}\leq\frac{\|(-\q)_+\|_\infty}{\beta(\mathcal{A})}$;
		\item[(ii)] $\|\x\|_2^{m-1}\leq\frac{\|(-\q)_+\|_2}{\mu(\mathcal{A})}$ if $\mathcal{A}$ is symmetric;
		\item[(iii)] $\|\x\|_m^{m-1}\leq\frac{\|(-\q)_+\|_\frac{m}{m-1}}{\lambda(\mathcal{A})}$ if $\mathcal{A}$ is symmetric,
	\end{itemize}
	where $\x_+=(\max\{x_1,0\},\max\{x_2,0\},\cdots,\max\{x_n,0\})^\top$.
\end{lemma}
Now we present lower bounds  of the solution set of TCP$(\mathcal{A},\q)$ when  $\mathcal{A} \in T_{m,n}$ is strictly semi-positive.
\begin{theorem}\label{t38}  Let $\mathcal{A} \in T_{m, n}$ ($m\geq2$) be strictly semi-positive, and let $\x$ be a solution of TCP$(\mathcal{A},\q)$.  Then\begin{itemize}
		\item[(i)] $\frac{\|(-\q)_+\|_\infty}{ n^{\frac{m-2}2}\|T_{\mathcal{A}}\|_\infty}\leq\|\x\|_\infty^{m-1}$;
		\item[(ii)] $\frac{\|(-\q)_+\|_\infty}{ \|F_{\mathcal{A}}\|_\infty^{m-1}}\leq\|\x\|_\infty^{m-1}$ if $m$ is even;
		\item[(iii)] $\frac{\|(-\q)_+\|_2}{ \|T_{\mathcal{A}}\|_2}\leq\|\x\|_2^{m-1}$;
		\item[(iv)] $\frac{\|(-\q)_+\|_m}{ \|F_{\mathcal{A}}\|_m^{m-1}}\leq\|\x\|_m^{m-1}$if $m$ is even.
	\end{itemize}
\end{theorem}
\begin{proof}  For $\q\geq\0$, it follows from Lemma \ref{le:25} that $\x=\0$. Since $\|(-\q)_+\|=0$, the conclusion holds obviously. Therefore, we may assume that $\x\ne \0$, or equivalently, that $\q$ is not nonnegative. Since $\x$ is a solution of TCP$(\mathcal{A},\q)$, we have \begin{equation} \label{eq:311}
	\x \geq \0, \q + \mathcal{A}\x^{m-1} \geq \0, \mbox{ and }\x^\top (\q + \mathcal{A}\x^{m-1}) = 0,
	\end{equation}
	and hence, $$(\mathcal{A} \x^{m-1})_i\geq (-\q)_i \mbox{ for all } i\in I_n.$$
	In particular, $$|(\mathcal{A} \x^{m-1})_i|\geq((\mathcal{A} \x^{m-1})_+)_i\geq ((-\q)_+)_i\mbox{ for all } i\in I_n.$$	
	Thus,\begin{equation} \label{eq:312}\|\mathcal{A} \x^{m-1}\|_\infty\geq\|(-\q)_+\|_\infty.\end{equation}
	By the above inequality together with Lemma \ref{le:23}, we have  \begin{align}
	\|(-\q)_+\|_\infty\leq& \|\x\|_2^{m-2}\|\|\x\|_2^{2-m}\mathcal{A} \x^{m-1}\|_\infty\nonumber\\
	=&\|\x\|_2^{m-2}\|T_{\mathcal{A}}(\x)\|_\infty\nonumber\\
	\leq& \|\x\|_2^{m-2}\|\x\|_\infty\|T_{\mathcal{A}}\|_\infty\nonumber\\
	\leq& n^{\frac{m-2}2}\|\x\|_\infty^{m-1}\|T_{\mathcal{A}}\|_\infty\mbox{\ (use $\|\x\|_2\leq \sqrt{n}\|\x\|_\infty$)}.\nonumber
	\end{align}
	This prove (i).  Next we show (ii). Similarly, using Lemma \ref{le:23} and (\ref{eq:312}), we also have
	\begin{align}
	\|(-\q)_+\|_\infty\leq& \|(\mathcal{A} \x^{m-1})^{[\frac1{m-1}]}\|_\infty^{m-1}\nonumber\\
	=&\|F_{\mathcal{A}}(\x)\|_\infty^{m-1}\nonumber\\
	\leq& \|\x\|_\infty^{m-1}\|F_{\mathcal{A}}\|_\infty^{m-1}.\nonumber
	\end{align}
	(iii) It follows from (\ref{eq:312}) and Lemma \ref{le:23} that
	\begin{align}
	\|(-\q)_+\|_2\leq&\|\mathcal{A} \x^{m-1}\|_2= \|\x\|_2^{m-2}\|\|\x\|_2^{2-m}\mathcal{A} \x^{m-1}\|_2\nonumber\\
	=&\|\x\|_2^{m-2}\|T_{\mathcal{A}}(\x)\|_2\nonumber\\
	\leq& \|\x\|_2^{m-2}\|\x\|_2\|T_{\mathcal{A}}\|_2\nonumber\\
	\leq& \|\x\|_2^{m-1}\|T_{\mathcal{A}}\|_2.\nonumber
	\end{align} So,(iii) is proved. Now we show (iv).  It follows from (\ref{eq:312}) and  Lemma \ref{le:23} that \begin{align}
	\|(-\q)_+\|_m\leq&\|\mathcal{A} \x^{m-1}\|_m\leq\|\mathcal{A} \x^{m-1}\|_\frac{m}{m-1}= \|(\mathcal{A} \x^{m-1})^{[\frac1{m-1}]}\|_m^{m-1}\nonumber\\
	=&\|F_{\mathcal{A}}(\x)\|_m^{m-1}\nonumber\\
	\leq& \|\x\|_m^{m-1}\|F_{\mathcal{A}}\|_m^{m-1}.\nonumber
	\end{align}
	The desired inequality follows.
\end{proof}
Combining Lemmas \ref{le:23} and \ref{le:31} with Theorem \ref{t38}, the following theorems are easily proved.

\begin{theorem}\label{t39}  Let $\mathcal{A} \in T_{m, n}$ ($m\geq2$) be strictly semi-positive, and let $\x$ be a solution of TCP$(\mathcal{A},\q)$.  Then
	\begin{equation}\label{eq:313}
	\frac{\|(-\q)_+\|_\infty}{ n^{\frac{m-2}{2}}\max\limits_{i\in I_n}\left(\sum\limits_{i_2,\cdots,i_m=1}^{n}|a_{ii_2\cdots i_m}|\right)}\leq\|\x\|_\infty^{m-1}\leq \frac{\|(-\q)_+\|_\infty}{\beta({\mathcal{A}})}.
	\end{equation}
\end{theorem}
\begin{theorem}\label{t310}  Let $\mathcal{A} \in T_{m, n}$ ($m\geq2$) be strictly semi-positive, and let $\x$ be a solution of TCP$(\mathcal{A},\q)$. If $m$ is even, then
	\begin{equation}\label{eq:314}
	\frac{\|(-\q)_+\|_\infty}{\max\limits_{i\in I_n}\left(\sum\limits_{i_2,\cdots,i_m=1}^{n}|a_{ii_2\cdots i_m}|\right)}\leq\|\x\|_\infty^{m-1}\leq \frac{\|(-\q)_+\|_\infty}{\beta({\mathcal{A}})}.
	\end{equation}
\end{theorem}

Combining Lemmas \ref{le:24} and \ref{le:31} with Theorem \ref{t38}, the following theorems are easily proved.

\begin{theorem}\label{t311}  Let $\mathcal{A} \in T_{m, n}$ ($m\geq2$) be strictly semi-positive, and let $\x$ be a solution of TCP$(\mathcal{A},\q)$.   if $\mathcal{A}$ is symmetric, then
	\begin{equation}\label{eq:315}
	\frac{\|(-\q)_+\|_2}{n^{\frac{m-2}2}\left(\sum\limits_{i=1}^{n}\left(\sum\limits_{i_2,\cdots,i_m=1}^{n}|a_{ii_2\cdots i_m}|\right)^2\right)^{\frac12}}\leq\|\x\|_2^{m-1}\leq \frac{\|(-\q)_+\|_2}{\mu(\mathcal{A})}
	\end{equation}
\end{theorem}
\begin{theorem}\label{t312}  Let $\mathcal{A} \in T_{m, n}$ ($m\geq2$) be strictly semi-positive, and let $\x$ be a solution of TCP$(\mathcal{A},\q)$. If $\mathcal{A}$ is symmetric and $m$ is even, then
	\begin{equation}\label{eq:316}
	\frac{\|(-\q)_+\|_m}{\left(\sum\limits_{i=1}^{n}\left(\sum\limits_{i_2,\cdots,i_m=1}^{n}|a_{ii_2\cdots i_m}|\right)^{\frac{m}{m-1}}\right)^{\frac1m}}\leq\|\x\|_m^{m-1}\leq \frac{\|(-\q)_+\|_\frac{m}{m-1}}{\lambda(\mathcal{A})}.
	\end{equation}
\end{theorem}
When $m=2$, both $\lambda (A)$ and $\mu(A)$ are the smallest Pareto eigenvalue of  a matrix $A$, denote by $\lambda (A)$. For more details on Pareto eigenvalue of  a matrix, see Seeger \cite{S99}, Seeger, Torki\cite{ST03} and Hiriart-Urruty, Seeger \cite{HS10}. Then the following conclusions are easy to obtain.
\begin{corollary}\label{c313}  Let $A$ be a strictly semi-monotone $n\times n$ matrix, and let $\x$ be a solution of LCP$(A,\q)$. Then
	\begin{itemize}
		\item[(i)] $\frac{\|(-\q)_+\|_\infty}{\max\limits_{i\in I_n}\sum\limits_{j=1}^{n}|a_{ij}|}\leq\|\x\|_\infty\leq \frac{\|(-\q)_+\|_\infty}{\beta (A)}$;
		\item[(ii)]	$\frac{\|(-\q)_+\|_2}{\left(\sum\limits_{i=1}^{n}\left(\sum\limits_{j=1}^{n}|a_{ij}|\right)^2\right)^{\frac12}}\leq\|\x\|_2\leq \frac{\|(-\q)_+\|_2}{\lambda (A)}$  if $A$ is symmetric.	
	\end{itemize}	
\end{corollary}

From Lemma \ref{le:21} and Theorems  \ref{t39}, \ref{t310}, \ref{t311}, \ref{t312}, the following corollary follows.
\begin{corollary}\label{c314}  Let $\mathcal{A} \in S_{m, n}$ ($m\geq2$) be strictly copositive, and let $\x$ be a solution to TCP$(\mathcal{A},\q)$.  Then
	\begin{itemize}
		\item[(i)] 	$\frac{\|(-\q)_+\|_\infty}{ n^{\frac{m-2}{2}}\max\limits_{i\in I_n}\left(\sum\limits_{i_2,\cdots,i_m=1}^{n}|a_{ii_2\cdots i_m}|\right)}\leq\|\x\|_\infty^{m-1}\leq \frac{\|(-\q)_+\|_\infty}{\beta({\mathcal{A}})}$;
		\item[(ii)]	$\frac{\|(-\q)_+\|_\infty}{\max\limits_{i\in I_n}\left(\sum\limits_{i_2,\cdots,i_m=1}^{n}|a_{ii_2\cdots i_m}|\right)}\leq\|\x\|_\infty^{m-1}\leq \frac{\|(-\q)_+\|_\infty}{\beta({\mathcal{A}})}$  if $m$ is even;
		\item[(iii)] 	$\frac{\|(-\q)_+\|_2}{n^{\frac{m-2}2}\left(\sum\limits_{i=1}^{n}\left(\sum\limits_{i_2,\cdots,i_m=1}^{n}|a_{ii_2\cdots i_m}|\right)^2\right)^{\frac12}}\leq\|\x\|_2^{m-1}\leq \frac{\|(-\q)_+\|_2}{\mu(\mathcal{A})}$;
		\item[(iv)]	$	\frac{\|(-\q)_+\|_m}{\left(\sum\limits_{i=1}^{n}\left(\sum\limits_{i_2,\cdots,i_m=1}^{n}|a_{ii_2\cdots i_m}|\right)^{\frac{m}{m-1}}\right)^{\frac1m}}\leq\|\x\|_m^{m-1}\leq \frac{\|(-\q)_+\|_\frac{m}{m-1}}{\lambda(\mathcal{A})}$  if $m$ is even;
	\end{itemize}	
\end{corollary}
\begin{question}\label{Q2}  We obtain the upper and lower bounds of solution set for tensor complementarity problem (TCP) with strictly semi-positive tensors (Theorems  \ref{t39}, \ref{t310}, \ref{t311}, \ref{t312}).
	\begin{itemize}
		\item Are the upper and lower bounds best?
		\item May the symmetry be removed in Theorems \ref{t311} and \ref{t312}?
		\item How to design an effective algorithm to compute the bounds?
	\end{itemize}
\end{question}

\section{Conclusions}
In this paper,  we discuss some nice properties of strictly semi-positive tensors. The quantity $\beta(\mathcal{A})$ closely adheres to strictly semi-positive tensors and play an important role in the error bound analysis of TCP($\mathcal{A},\q$).   More specifically,  the following conclusions are proved.
\begin{itemize}
	\item  The monotonicity and boundedness of $\beta(\mathcal{A})$ are established;
	\item  A tensor $\mathcal{A}$ is strictly semi-positive if and only if $\beta(\mathcal{A})>0$;
	\item  Each $H^+(Z^+)$-eigenvalue of a strictly semi-positive tensor is strictly positive;
	\item  We introduce two quantities $\delta_{H^+}(\mathcal{A})$ and $\delta_{Z^+}(\mathcal{A})$, and show they closely adhere to the upper bounds of the constant $\beta(\mathcal{A})$;
	\item  The boundedness of solution set are presented for TCP($\mathcal{A},\q$) with a strictly semi-positive tensor $\mathcal{A}$.
\end{itemize}

\section*{Acknowledgment} The authors would like to thank  Editor, the anonymous referees  for their valuable suggestions which helped us to improve this manuscript. Our work was  supported by
the National Natural Science Foundation of P.R. China (Grant No.
11571095, 11601134) and by the Hong
Kong Research Grant Council (Grant No. PolyU
502111, 501212, 501913 and 15302114). This work was partially done when the first author was
visiting The Hong Kong Polytechnic University.



\end{document}